\documentclass{article}
\usepackage{geometry}                % See geometry.pdf to learn the layout options. There are lots.
\usepackage{mathtools}
\usepackage{amssymb}
\usepackage{amsmath}
\usepackage{amsthm,dsfont,yfonts}
\usepackage{rotating}
\usepackage{graphicx,pspicture}
\usepackage{array}
\usepackage{calc}
\usepackage{soul}
\usepackage[all]{xy}
\usepackage{amsthm}
\usepackage{lmodern}
\usepackage{color}
\usepackage{pstricks,pst-node,pst-tree}
\usepackage{graphics,graphicx}
\usepackage[british]{babel}
\usepackage{fancyhdr}
\usepackage{color}
\usepackage{pinlabel}
\usepackage{enumerate}
\usepackage{footmisc}

\newtheorem{theorem}{Theorem}[section]
\newtheorem{cor}[theorem]{Corollary}
\newtheorem{lemma}[theorem]{Lemma}
\newtheorem{prop}[theorem]{Proposition}

\theoremstyle{definition}

\newtheorem{example}[theorem]{Example}

\definecolor{magenta}{RGB}{203,0,150}
\definecolor{blueish}{RGB}{0,35,211}

\newcommand{\h}{\hspace{2mm}}  %for white space in the mathematical environment

\usepackage{makeidx}
\makeindex

\usepackage[noindentafter]{titlesec}
\DefineFNsymbols*{lamportnostar}[math]{\wr|{\wr}{\wr}}
\setfnsymbol{lamportnostar}

\title{Palindromic Width of Wreath Products}
\date{\today}
\author{Elisabeth Fink\footnote{This work is supported by the ERC starting grant 257110 ``RaWG''}}
  
\begin{document}

\selectlanguage{british}

\maketitle

\begin{abstract}
We show that the wreath product $G \wr \mathbb{Z}^n$ of any finitely generated group $G$ with $\mathbb{Z}^n$ has
finite palindromic width. We also show that $C \wr A$ has finite palindromic width if $C$ has finite commutator width
and $A$ is a finitely generated infinite abelian group. Further we prove that if $H$ is a non-abelian group with
finite palindromic width and $G$ any finitely generated group, then every element of the subgroup $G' \wr H$ can be expressed as a
product of uniformly boundedly many palindromes. From this we obtain that $P \wr H$ has finite palindromic width if $P$ is a
perfect group and further that $G \wr F$ has finite palindromic width for any finite, non-abelian group $F$.
\end{abstract}

\section{Introduction}

Palindromic words in groups have been studied from various angles lately. They make their first appearance in
\cite{collins_palindromes}, where D. Collins studied palindromic automorphisms of free groups. In \cite{glover_jensen} H.H. Glover
and C.A. Jensen study the geometry of palindromic automorphism groups of the free group. Later in
\cite{bardakov_shpilrain_tolstykh} it was shown that free groups have infinite palindromic and primitive width and F. Deloup
\cite{deloup} studied the palindromic map, which is an anti-automorphism, in braid and Artin groups.

\medskip

More recently, it has been established by V. Bardakov and K. Gongopadhyay \cite{bardakov} that free nilpotent groups and free
abelian-by-nilpotent groups have finite palindromic width. A paper in preparation by the same authors shows that some extensions
and quotients of these groups have finite palindromic width as well \cite{bardakov_nilpotent}. Another paper
\cite{bardakov_soluble} by the same authers proves that certain soluble groups have finite palindromic width. In \cite{bardakov}
they use results about
the commutator width in nilpotent groups to establish their result. Independently, it has been shown by T. Riley
and A. Sale in \cite{rileySale} that free metabelian groups have finite palindromic width by using results about skew-symmetric
functions on free abelian groups. Further, the same authors show that $B \wr \mathbb{Z}^n$ has finite palindromic width if $B$ is
a group with finite palindromic width.

\medskip

We extend this result to the case where $G$ is any finitely generated group, then we show that $G \wr \mathbb{Z}^n$ has
finite palindromic width.  A result by M. Akhavan-Malayeri \cite{akhavan_comm_wreath} shows that the wreath product of $F_d$ with
$\mathbb{Z}^n$ has finite commutator width. We use the result from \cite{akhavan_comm_wreath} to prove that $F_d \wr \mathbb{Z}^n$
has finite
palindromic width and then deduce that this property also holds for its quotients. We also give a proof that the
wreath product $C \wr A$ has finite palindromic width if $C$ is a finitely generated group that has finite commutator width and
$A$ a finitely generated infinite abelian group.

\medskip

More generally, let $G$ be any finitely generated group and $H$ a non-abelian group which has finite palindromic width with
respect to some generating set. We establish that every element of the subgroup $G' \wr H$ of the regular wreath product $G \wr H$
is a finite product of palindromes. As a Corollary we obtain that $G \wr F$ has finite palindromic width if $F$ is a non-abelian
finite group.

\medskip

More concretely, we prove the following, where $pw(G,X)$ denotes the palindromic width of the group $G$ with respect to the
generating set $X$.

\begin{theorem}
\begin{enumerate}
 \item Let $G$ be a $d$-generated group generated by $X$ and $E$ be the standard generating set of $\mathbb{Z}^k$, for $k \in
\mathbb{N}$. Then we have that $pw\left(G \wr \mathbb{Z}^k, X \cup E\right) \leq 5d+9k$ if $k$ is even and $pw\left(G \wr
\mathbb{Z}^k, X \cup E\right) \leq 5d+9k+2$ if $k$ is odd.
 \item Assume that $A$ is an $r$-generated infinite abelian group generated by $T$ and $C$ a finitely generated group with finite
commutator width $n$. Then $pw(C \wr A, T \cup Y) \leq 5|Y|+6r+7n$, for any finite generating set $Y$ of $C$.
 \item Let $G$ be a finitely generated group and $H$ a non-abelian finitely generated group with
finite palindromic width with respect to the generating set $Z$. Then every element of the subgroup $G' \wr H$ can be
written as a product of at most $pw(H,Z)+1$ palindromes.
 \item Assume that $P$ is a finitely generated perfect group and $H$ as above. Then for any generating set $S$ of $P$ we have
$pw(P \wr H, Z \cup \{c\} \cup S) \leq pw(H,Z)+1$, where $c \in H$, possibly $c=1$.
 \item If $K$ is a non-abelian finite group and $G$ a $d$-generated group, then $G \wr H$ has finite palindromic width with
respect to the natural generating set.
\end{enumerate}

\end{theorem}

\textbf{Acknowledgement.} I would like to thank Tim Riley for many helpful comments and for pointing out a crucial mistake in an
earlier version of this paper. 

\section{Wreath Products and Words}

The regular wreath product $B \wr A$ of two groups $B$ and $A$ is given by $F \rtimes A$, where $F = \prod_A B$ and the action of
$A$ on itself is given by left multiplication. We can write every $g \in B \wr A$ as
\begin{equation}\label{eq_g1}
g = a \cdot \prod_{i=1}^k a_i^{-1} \left(f_i,1\right) a_i
\end{equation}
with $a,a_i \in A$ and $f_i
\in B$ and $k \in \mathbb{Z}$. If $X$ is a generating set for $A$ and $Y$ one for $B$, then the set $\left\{\left(y_i,1\right),
\left(1,x_i\right) \h | \h y_i \in Y, x_i \in X\right\}$ is a natural generating set for $B \wr A$.

\medskip

A \emph{group word} $w$ is an element of $F_n=\left<x_1, \dots, x_n\right>$, the free group on $n$ generators. Examples of such
words are commutators $w\left(x_1, x_2\right) = x_1^{-1}x_2^{-1}x_1x_2$, Burnside words $w\left(x_1\right)=x_1^p$ for some $p \in
\mathbb{Z}$ and many others. A word $w$ has the general form $w\left(x_1, \dots, x_n\right)=\prod_{j=1}^k x_{i_j}^{p_j}$ for $k
\in \mathbb{N}$ and $p_j \in \mathbb{Z}$. We say a word $w$ is a \emph{palindrome} if $x_{i_j}=x_{i_{k-j+1}}$ and $p_j=p_{k-j+1}$
for all $j=1, \dots, k$. Examples of palindromic words are $x_1x_2x_1$, all power words or words of the form $x_1^2x_2x_1^2$.
For any word $w$ in $F_n$ we denote by $\overline{w}$ the word 
\[\overline{w} = \prod_{j=1}^k x_{i_{k-j+1}}^{p_{k-j+1}},\] which we call the \emph{reverse word} of $w$. We will use this
notation throughout this paper.

\medskip

Every word $w \in F_n$ defines a verbal mapping for any group $G$ from $G^{(k)}=\prod_{i=1}^k G$ to $G$ in a natural way. More
details on this can be found in \cite{dan_words}. We call the image \[w\left(G, \dots, G\right) = \left<\left\{w\left(g_1,\dots,
g_n\right) \mid g_i \in G\right\}\right>\] the \emph{verbal subgroup} of $w$ in $G$. It is a well-studied topic to decide when
this
group is actually generated by products of at most $C \in \mathbb{N}$ words. If there exists such a constant $C$, then we say that
the word $w$
has \emph{finite width} in $G$. The most studied word is the commutator
word, leading to the question when every element of $G'$ can be expressed as a product of at most $C$ commutators. Examples of
groups having finite commutator width are all finite simple groups \cite{ore_conj} with commutator width $1$ or most finite
quasi-simple groups \cite{quasi_simple_groups}. Infinite examples are soluble groups satisfying the maximal condition on normal
subgroups \cite{akhavan_2}. It has been shown in \cite[Theorem 2.1.3]{dan_words} that any virtually nilpotent group of finite rank
has the property that any word has finite width. We note that palindromes are not group words under this definition.

\medskip

However, in analogy with the terminology above, we say that an element $g$ of a group is a palindrome, if it can be represented
by a palindromic word in the generators of $G$. Further, a group $G$ has \emph{finite palindromic width} $pw(G,X)$ with respect
to a generating set $X$, if every element of $G$ can be expressed as a product of at most $pw(G,X)$ many palindromes in the
alphabet $X$. If we just say a group $G$ has finite palindromic width without reference to a generating set, then we mean that
there exists a generating set with respect to which $G$ has finite palindromic width. It is yet unclear how the choice of the
generating set influences the palindromic width of $G$. We will however use below that if $G$ has finite palindromic width with
respect to a set $X$, then it has at most the same finite palindromic width with respect to any generating set containing $X$.

%%%%%%%%%%%%%%%%%%%%%%%%%%%%%%%%%%%%%%%%%%%%%%%%%%%%%%%%%%%%%%%%%%%%%%%%%%%%%%%%%%%%%%%%%%%%%%%%%%%%%%%%%%%%%%
%%%%%%%%%%%%%%%%%%%%%%%%%%%%%%%%%%%%%%%%%%%%%%%%%%%%%%%%%%%%%%%%%%%%%%%%%%%%%%%%%%%%%%%%%%%%%%%%%%%%%%%%%%%%%%
%%%%%%%%%%%%%%%%%%%%%%%%%%%%%%%%%%%%%%%%%%%%%%%%%%%%%%%%%%%%%%%%%%%%%%%%%%%%%%%%%%%%%%%%%%%%%%%%%%%%%%%%%%%%%%
%%%%%%%%%%%%%%%%%%%%%%%%%%%%%%%%%%%%%%%%%%%%%%%%%%%%%%%%%%%%%%%%%%%%%%%%%%%%%%%%%%%%%%%%%%%%%%%%%%%%%%%%%%%%%%
%%%%%%%%%%%%%%%%%%%%%%%%%%%%%%%%%%%%%%%%%%%%%%%%%%%%%%%%%%%%%%%%%%%%%%%%%%%%%%%%%%%%%%%%%%%%%%%%%%%%%%%%%%%%%%
\section{Wreath Product with a Free Abelian Group}

In this Section we use the results of M. Akhavan-Malayeri from \cite{akhavan_comm_wreath} and V. Bardakov and K. Gongopadhyay
\cite{bardakov} to deduce that the wreath product of a finitely generated group with a finitely generated free abelian group has
finite palindromic width. We also give a self-contained proof that $C \wr A$ has finite palindromic width if $C$ is a group with
finite commutator width and $A$ an infinite finitely generated abelian group.

\medskip

We will implicitly use the following Lemma in many places. The proof of it is obvious and we therefore omit it.

\begin{lemma}\label{lemma_overline}
\begin{enumerate}[(a)]
 \item If $A$ is an abelian group, then $a=\overline{a}$ for all $a \in A$. 
 \item For any group $G$ and every element $g \in G$ we have that $\left(\overline{g}\right)^{-1} = \overline{g^{-1}}$.
\end{enumerate}
\end{lemma}

\medskip

The following Theorem gives a uniform bound on the commutator width of a wreath product. A first, less general, version of this
can also be found in M. Akhavan-Malayeri's thesis \cite{akhavan_malayeri_thesis}.

\begin{theorem}[\cite{akhavan_comm_wreath}]\label{thm_akhavan_wreath}
Let $A$ be a non-abelian free group and $W=A \wr B$, where $B$ is a free abelian group of rank $n$. Then every element
of $W'$ is a product of at most $n+2$ commutators, and furthermore $2 \leq cw(W) \leq n+2$. 
\end{theorem}

We use the shape of the commutators that emerges in the proof of Theorem \ref{thm_akhavan_wreath} to prove that the same wreath products also have finite palindromic width. The following result from \cite{bardakov} gives an upper bound for the palindromic width of a metabelian group.

\begin{theorem}\cite{bardakov}\label{thm_abelian}
Let $G$ be a non-abelian free abelian-by-nilpotent group of rank $n$. Then $pw(G,X) \leq 5n$ for a minimal generating set $X$.
\end{theorem}

Combining these two Theorems allows us to obtain a bound for the palindromic width of such wreath products.

\begin{theorem}\label{thm_wreath_prod_abelian}
The wreath product $F_d \wr \mathbb{Z}^n$ of the free $d$-generated group $F_d$ with $\mathbb{Z}^n$ has finite
palindromic width at most $5d+9n$ if $n$ is even and at most $5d+9n+2$ if $n$ is odd with respect to the natural generating set
coming from $F_d$ and $\mathbb{Z}^n$.
\end{theorem}

\begin{proof}
Set $B = \prod_{\mathbb{Z}^n} F_d$ to be the base group of the wreath product. We can split the wreath product into $(F_d^{ab} \wr
A) \cdot B'$. The first group $F_d^{ab} \wr \mathbb{Z}^n$ is a metabelian group which has finite palindromic width at most
$5(d+n)$ by Theorem \ref{thm_abelian}. In \cite{akhavan_comm_wreath} the author proves that every element $w$ of $B'$ is of the
form
\[w = [a,t][b,t^2], \quad \mbox{with} \quad a,b \in B,t \in \mathbb{Z}^n.\] Since $t$ is an element of an abelian group, it
can be written as the product of finitely many palindromes, $t = t_1^{i_1}\cdot \dots \cdot t_n^{i_n}$ with $i_j \in \mathbb{Z}$,
where $\left\{t_1, \dots, t_n\right\}$ is the finite generating set of $\mathbb{Z}^n$. This implies that $t^2= t_1^{2i_1}\cdot
\dots \cdot t_n^{2i_n}$. First assume that $n$ is even. Then we can write the first commutator as
\[[a,t_1^{i_1}\cdot \dots \cdot t_n^{i_n}] = a^{-1} (t_1^{i_1}\cdot \dots \cdot t_n^{i_n})^{-1} a t_1^{i_1}\cdot \dots \cdot t_n^{i_n}\]
\[= a^{-1} t_n^{-i_n} \overline{a^{-1}} \cdot \overline{a} t_{n-1}^{-i_{n-1}} a \cdot a^{-1} t_{n-2}^{-i_{n-2}} \overline{a^{-1}} \cdot \dots \cdot \overline{a} t_{1}^{-i_{1}} a \cdot t_1^{i_1} \cdot \dots \cdot t_n^{i_n}.\] 
This is now a product of $2n$ palindromes. If $n$ is odd, we write
\[[a,t_1^{i_1}\cdot \dots \cdot t_n^{i_n}] = a^{-1} (t_1^{i_1}\cdot \dots \cdot t_n^{i_n})^{-1} a t_1^{i_1}\cdot \dots \cdot
t_n^{i_n}\]
\[= a^{-1} t_n^{-i_n} \overline{a^{-1}} \cdot \overline{a} t_{n-1}^{-i_{n-1}} a \cdot a^{-1} t_{n-2}^{-i_{n-2}} \overline{a^{-1}}
\cdot \dots \cdot a^{-1} t_{1}^{-i_{1}} \overline{a^{-1}} \cdot \overline{a}a \cdot t_1^{i_1} \cdot \dots \cdot t_n^{i_n},\] 
which is a product of $2n+1$ palindromes. We can do the same for the second commutator, hence we need at most
$4n$ if $n$ is even and at most $4n+2$ palindromes if $n$ is odd to express every element of $B'$. In total we need at
most $5(d+n)+4n=5d+9n$ palindromes if $n$ is even and at most $5d+9n+2$ palindromes if $n$ is odd. These are palindromes with
respect to the natural generating set of $F_d \wr \mathbb{Z}^n$.
\end{proof}

The property of a group of having finite palindromic width with respect to some generating set carries over to its quotients.

\begin{lemma}\label{lemma_pw_in_quotients}
Let $G$ be a group with finite palindromic width $pw(G,X)$ with respect to $X$. If $N$ is a normal subgroup of $G$, then
$pw(G/N,\hat{X}) \leq pw(G,X)$ where $\hat{X}$ denotes the image of $X$ under the natural homomorphism $G\rightarrow G/N$.
\end{lemma}

Every finitely generated group can be obtained as a quotient of some free group. This now implies that Theorem
\ref{thm_wreath_prod_abelian} in fact holds for any wreath product $G \wr \mathbb{Z}^n$, where $G$ is a finitely generated group.

\begin{cor}
The wreath product $G \wr A$ of any $d$-generated group $G$ with a finitely $n$-generated abelian group $A$ has finite palindromic
width at most $5d+9n$ if $n$ is even and at most $5d+9n+2$ if $n$ is odd with respect to the natural generating set coming from
$A$ and $G$.
\end{cor}

We note however, that this does not imply that $G \wr A$ has finite palindromic width if $A$ is not free abelian. In fact,
consider $Q =
(G \wr A)/\left<\left<N\right>\right>^{G \wr A}$, where $N \lhd A$ and $\left<\left<N\right>\right>$ denotes the normal closure of
$N$ in $G \wr A$. Then for any non-trivial normal subgroup $N \lhd A$, the quotient $Q$ becomes metabelian.

\medskip

If we assume that the group $B$ has finite commutator width and $A$ is an infinite finitely generated abelian group, then it is
rather easy to construct palindromes without using the presentation of the commutators from \cite{akhavan_comm_wreath}. Hence
this result also holds if $A$ is not free.

\begin{theorem}
Let $C$ be a $d$-generated group with finite commutator width $cw(C)=n$ and $A$ an infinite abelian group of rank $r$ with
generating set $X$ such that $|X|=r$. Then \[pw(C \wr A, X \cup S) \leq 5d+6r + 7n\] for any chosen generating set $S$
of $C$.
\end{theorem}

\begin{proof}
We write $C \wr A$ as $\left(C^{ab} \wr A\right) \cdot \left(C'\wr A\right)$. The first factor has finite
palindromic width at most $5(d+r)$ by Theorem \ref{thm_abelian}. For $C' \wr A$ we do something similar as in the proof of
Theorem \ref{thm_der_group}. We want to write an element $g$ as a product of finitely many palindromes. Similar to \eqref{eq_g1},
assume $g$ is given as the element
\[g= a \cdot \prod_{i=1}^k a_i^{-1} b_i a_i, \quad b_i = \prod_{j=1}^{n_i} \left[f_{i,j},g_{i,j}\right], \quad n_i \in \mathbb{Z},
b_i, f_i, g_i \in C \times 1.\]
We have to find two elements $s,t \in A$ such that $a_i \neq a_i^{-1}s$ and $a_i \neq a_i^{-1}t
\neq a_i^{-1}s$. Since $A$ is infinite abelian this can always be done with $s$ and $t$ both being a different high power of
some infinite order generator $x \in X$. Assume $s = x^q$ and $t = x^y$ for some $q,y \in \mathbb{Z}$. We rewrite $g$ with
\[\kappa_j = \prod_{i=1}^k a_{i,j}^{-1} f_{i,j} a_{i,j}, \quad \tau_j = \prod_{i=1}^k a_{i,j}^{-1} g_{i,j} a_{i,j}, \quad
h = \prod_{j=1}^n \kappa_j^{-1} \tau_j^{-1} \kappa_j \tau_j\]
as $g = a\cdot h$. We can also write $g$ as 
\begin{equation}\label{eq_comm_fin}
g = a\cdot  \prod_{j=1}^n \kappa_j^{-1} s^{-1} \overline{\kappa_j^{-1}} \cdot s \cdot \tau_j^{-1} t^{-1}
\overline{\tau_j^{-1}} \cdot t \cdot s^{-1} \cdot \overline{\kappa_j} s \kappa_j \cdot t^{-1} \cdot \overline{\tau_j} t
\tau_j.\end{equation} 
The element $ts^{-1}$ is given by $x^{y-q}$, hence it is a palindrome. We see that this procedure needs at
most $r$ palindromes for $a$ and $7n$ palindromes for the product of the commutators in \eqref{eq_comm_fin}. Together we can write
every element of $C' \wr H$ under the given hypothesis as a product of at most $r+7n$ palindromes. With the $5(d+r)$ palindromes
from $C^{ab} \wr A$ we hence need at most $5(d+r)+r+7n=5d+6r+7n$ palindromes to express every element of $C \wr A$.
\end{proof}

%%%%%%%%%%%%%%%%%%%%%%%%%%%%%%%%%%%%%%%%%%%%%%%%%%%%%%%%%%%%%%%%%%%%%%%%%%%%%%%%%%%%%%%%%%%%%%%%%%%%%%%%%%%%%%
%%%%%%%%%%%%%%%%%%%%%%%%%%%%%%%%%%%%%%%%%%%%%%%%%%%%%%%%%%%%%%%%%%%%%%%%%%%%%%%%%%%%%%%%%%%%%%%%%%%%%%%%%%%%%%
%%%%%%%%%%%%%%%%%%%%%%%%%%%%%%%%%%%%%%%%%%%%%%%%%%%%%%%%%%%%%%%%%%%%%%%%%%%%%%%%%%%%%%%%%%%%%%%%%%%%%%%%%%%%%%
%%%%%%%%%%%%%%%%%%%%%%%%%%%%%%%%%%%%%%%%%%%%%%%%%%%%%%%%%%%%%%%%%%%%%%%%%%%%%%%%%%%%%%%%%%%%%%%%%%%%%%%%%%%%%%
%%%%%%%%%%%%%%%%%%%%%%%%%%%%%%%%%%%%%%%%%%%%%%%%%%%%%%%%%%%%%%%%%%%%%%%%%%%%%%%%%%%%%%%%%%%%%%%%%%%%%%%%%%%%%%
\section{Wreath Product with a Non-Abelian Group}

Let $H$ be a $d$-generated group and $G$ a group with finite palindromic width with respect to a generating set $X$. In this
Section we show that every element of the subgroup $H' \wr G$ is a product of uniformly bounded finitely many palindromes with
respect to a chosen generating set of $H \wr G$. As a Corollary we deduce that $G \wr K$ has finite palindromic width for all
non-abelian finite groups $K$.

\medskip

Our main Theorem uses that we can find a word $r$ in the generators of $G$ such that $r=1$, but $\overline{r}\neq 1$. As we will
see in Lemma \ref{lemma_change_gen_set}, this can be achieved in any non-abelian group of finite palindromic width by eventually
adding one element to the generating set. The following example demonstrates that this condition on a word $r$ is easily fulfilled
by many groups.

\begin{example}\label{ex_BS}
Consider the group $BS(n,m) = \left<a,b\h \mid \h a^{-1}b^n a = b^m\right>$. Clearly $a^{-1}b^nab^{-m}=1$, but $b^{-m}ab^na^{-1}
\neq 1$.
\end{example}

\begin{lemma}\label{lemma_change_gen_set}
Assume $G$ is a finitely generated group with finite palindromic width with respect to the generating set $X$. Then $G$ has a
presentation with generating set $X \cup \{c\}$, in which there exists a relation $q$ in $G$ such that $\overline{q}\neq 1$ in
$G$, where $c \in G$ and possibly $c=1$.
\end{lemma}

\begin{proof}
Since $G$ has finite palindromic width, it cannot have free quotients. Hence for every pair of non-commuting generators $x$ and
$y$, there exists a relation $w$, in which $x$ and $y$ occur. Without loss of generality we can assume that $x$ and $y$ occur as
subword $xy$. If there is a subword $t$ between $x$ and $y$, then either take the first generator used in $t$
that does not commute with $x$, or otherwise swap $x$ and $t$ in $r$. So we have
\[r=w_1 xy w_2=1.\] If also $\overline{r} = \overline{w_2} yx \overline{w_1}=1$, then we introduce a new generator $c=xy$ and get
\[r = w_1 c w_2, \quad \overline{r} = \overline{w_2} c \overline{w_1} = \overline{w_2} xy \overline{w_1}.\]
Now assume that also under the new generating set $X \cup \left\{c\right\}$ we have that $\overline{r}=1$. This implies
\[\overline{w_2} yx \overline{w_1} = 1 = \overline{w_2} xy \overline{w_1},\] hence $x$ and $y$ commute, a contradiction.
\end{proof}

It is yet unclear how or if the property of having finite palindromic width depends on the chosen generating set. However, as we
will see in the following Proposition, enlarging the generating set can only reduce the palindromic width or leave it unchanged.

\begin{prop}
Let $G$ be a group which has finite palindromic width with respect to the generating set $X$. If $Y \subset G$ is any finite
subset of $G$, then $pw(G, X \cup Y) \leq pw(G,X)$.
\end{prop}

\begin{proof}
This follows immediately because the set of palindromes in $G$ with respect to $X \cup Y$ contains the set of palindromes of $G$
with respect to $X$.
\end{proof}

We will now see how we can use this to prove that at every element of the subgroup $F_d' \wr H$ can be expressed as a
product of finitely many palindromes.

\begin{theorem}\label{thm_der_group}
Let $F_d$ be a free $d$-generated group and $G$ a non-abelian group with $pw(G,X)=m$. Then every element of the subgroup $F_d' \wr
G$ is a product of at most $m+1$ palindromes with respect to $X \cup \left\{c\right\} \cup Y$, where $c \in G$, possibly $c=1$.
\end{theorem}

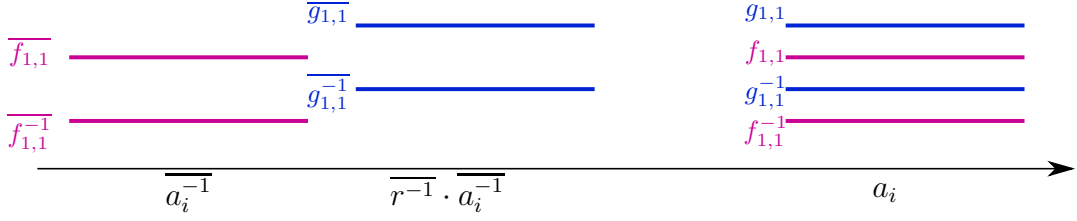
\begin{figure}[!h]
  \labellist
  \pinlabel \textcolor{magenta}{$\overline{f_{1,1}^{-1}}$} at 20 30
  \pinlabel \textcolor{blueish}{$\overline{g_{1,1}^{-1}}$} at 170 50
  \pinlabel \textcolor{magenta}{$\overline{f_{1,1}}$} at 20 70
  \pinlabel \textcolor{blueish}{$\overline{g_{1,1}}$} at 170 90
  
  \pinlabel \textcolor{magenta}{$f_{1,1}^{-1}$} at 390 30
  \pinlabel \textcolor{blueish}{$g_{1,1}^{-1}$} at 390 50
  \pinlabel \textcolor{magenta}{$f_{1,1}$} at 390 70
  \pinlabel \textcolor{blueish}{$g_{1,1}$} at 390 90
  
  \pinlabel{\large{$a_i$}} at 450 0
  \pinlabel{\large{$\overline{a_i^{-1}}$}} at 100 0
  \pinlabel{\large{$\overline{r^{-1}} \cdot \overline{a_i^{-1}}$}} at 230 0
  
  \endlabellist
  \includegraphics[scale=0.75]{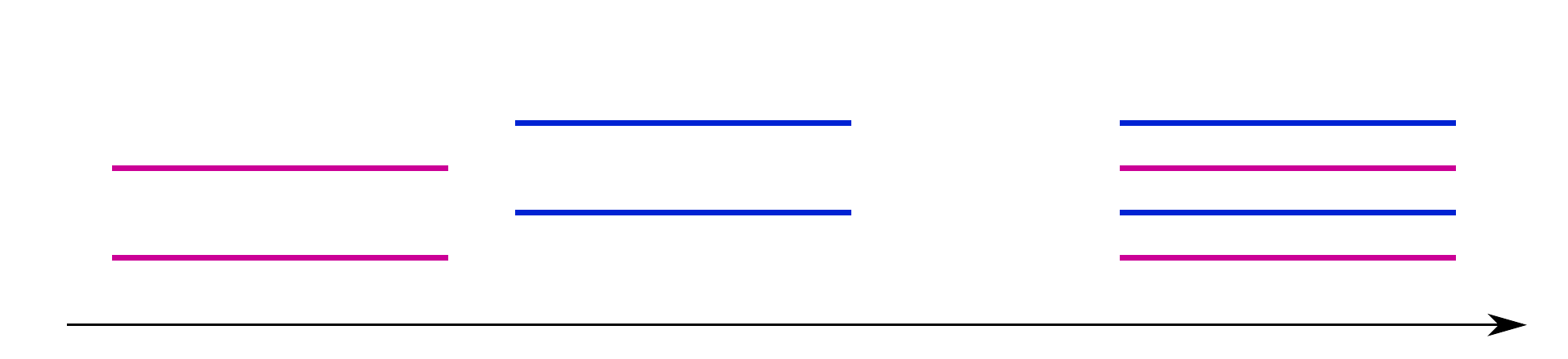}
  \caption{This is how we write one commutator on each position as a palindrome.}
  \label{fig_comm_1}
\end{figure}

The idea of the proof is that the commutator in Figure \ref{fig_comm_1} is a single palindrome of even length and in addition lies
in the base group of the wreath product. The elements on the right are the left hand side of the palindrome having positions $a_i$
and the elements on the left are the right hand side of the palindrome on positions $\overline{r^{-1}}\cdot \overline{a_i^{-1}}$
for a
relation $r$ that holds in $G$ such that $\overline{r}\neq 1$.

\begin{proof}[Proof of Theorem \ref{thm_der_group}]
Each element $b_i$ of $F_d' \times 1$ is a product of $n_i$ commutators. Assume that it is the product \[b_i = \prod_{j=1}^{n_i}
\left[f_{i,j}, g_{i,j}\right].\] Now an element $g$ of $F_d'\wr G$ has the form
\begin{equation}\label{eq_g} g=a \cdot \prod_{i=1}^k a_i^{-1} b_i a_i = a \cdot \prod_{i=1}^k a_i^{-1}
\left(\prod_{j=1}^{n_i}\left[f_{i,j},g_{i,j}\right]\right) a_i,\end{equation} with $f_i,g_i,b_i \in F_d \times 1$. We take into
account that we need at most $m$
palindromes for the element $a$ and concentrate on the product of commutators. Take $n$ to be the maximum over all $n_i$ and
write every $b_i$ as a product of $n$ (possibly trivial) commutators.

\medskip

We apply Lemma \ref{lemma_change_gen_set} to the generating set of $G$ and choose the relation $r$ in $G$ for which
$\overline{r}\neq 1$. We rewrite $g$ with 
\[\kappa_j = \prod_{j=1}^n a_{i}^{-1} f_{i,j}^{-1} r^{-1} g_{i,j}^{-1} r f_{i,j} r^{-1} g_{i,j} r a_{i},
\quad h = \prod_{i=1}^k \kappa_i,\] as $g= a \cdot h$. We see that now $\overline{h}$ is given by
\[\overline{h}= \prod_{i=k}^1 \overline{\kappa_i} = \prod_{i=k}^1 \overline{a_{i}} \cdot \overline{r} \cdot \overline{g_{i,j}}
\cdot \overline{r^{-1}} \cdot \overline{f_{i,j}} \cdot \overline{r} \cdot \overline{g_{i,j}^{-1}} \cdot \overline{r^{-1}} \cdot
\overline{f_{i,j}^{-1}} \cdot \overline{a_{i}^{-1}}.\]
We notice that for each $i$ the position of $\overline{f_{i,j}^{\pm 1}}$ is $\overline{a_{i}^{-1}}$ whereas all
$\overline{g_{i,j}^{\pm 1}}$ are at positions $\overline{r^{-1}} \cdot \overline{a_{i}^{-1}} $. This yields that
$\overline{\kappa_i}=1$ for each $i=1, \dots, n$ and hence $\overline{h}=1$. The word $h\overline{h}$ is of course a palindrome
and since $\overline{h}=1$ we have that $h\overline{h}=h$. We form the word
\[w_g = a h \overline{h},\] which is by construction representing $g$. We need at most $m$ palindromes to express $a$ and
$h\overline{h}$ is exactly one palindrome. Together we can write every element of $F_d' \wr G$ as a product of at most $m+1$
palindromes with respect to the generating set $X \cup \{c\} \cup Y$, with $c \in G$ being possibly trivial.
\end{proof}

We note that if $F_d^{ab} \wr G$ has finite palindromic width for a group $G$ with finite palindromic width, then $F_d \wr G$ has
finite palindromic width as well. However, it seems rather difficult to prove that $F_d^{ab} \wr G$ has finite palindromic width
in general for any non-abelian group $G$.

\medskip

We can again apply Lemma \ref{lemma_pw_in_quotients} and deduce that $H \wr G$ has finite palindromic width for any finitely
generated group $H$ and $G$ as in the Theorem above.

\begin{cor}
Let $H$ be a finitely generated generated group and $G$ a non-abelian group with $pw(G,X)=k$. Then every element of the
subgroup $H' \wr G$ is a product of at most $k+1$ palindromes with respect to $X \cup \left\{c\right\} \cup Y$, where $c \in G$,
possibly trivial.
\end{cor}

A group $P$ is called \emph{perfect} if $P=[P,P]$, in other words, if every element of $P$ is a product of commutators. This can
be applied to Theorem \ref{thm_der_group}.

\begin{cor}
Let $P$ be a finitely generated perfect group and $G$ a non-abelian finitely generated group with finite palindromic width
$pw(G,X)=k$ with respect to a generating set $X$. Then $P \wr G$ has finite palindromic width at most $k+1$.
\end{cor}

Examples of perfect groups are alternating groups, all simple groups and certain classes of branch groups as described in
\cite{Segal_subgroupGrowth} and in a forthcoming paper of this author. It has further been pointed out by D. Gruber that it is
possible to construct hyperbolic perfect groups using small cancellation. 

\medskip

If we assume $H$ to be a non-abelian finite group with generating set $X$, then it follows immediately that every element $h \in
H$ is a product of at most $\max\left\{l_X(h) \h | \h h \in H\right\}$ palindromes, where $l_X(h)$ denotes the word length of $h$
with respect to a finite generating set $X$. With this we can deduce that also the wreath product $F_d \wr H$ has finite
palindromic width.

\begin{theorem}
Assume $H$ is a finite non-abelian group with generating set $X$ and $F_d$ a free $d$-generated group.
Then \[pw\left(F_d \wr H, X \cup \{c\} \cup Y\right) \leq \max\left\{l_X(h) \h | \h h \in H\right\} \cdot (d \cdot |H| + 1) +
1,\] with $c
\in H$, possibly $c=1$.
\end{theorem}

\begin{proof}
We split the group $F_d \wr H$ into $\left(F_d^{ab} \wr H\right) \cdot \left(F_d' \wr H\right)$. The statement for $t \in F_d' \wr
H$ follows from Theorem \ref{thm_der_group}, using at most $\max\left\{l_X(h) \h | \h h \in H\right\} + 1$ palindromes in the
generating set $X \cup \left\{c\right\} \cup Y, c \in H$, to express $t$. An element $g$ of $F_d^{ab} \wr H$ has the form
\[g = a \cdot \prod_{i=1}^{|H|} c_i^{-1} b_i c_i, \quad c_i \in H, b_i \in F_d^{ab} \times 1\] or more convenient in this case
\[g = \prod_{i=1}^{|H|} a_i b_i, \quad a_i \in H, b_i \in F_d^{ab} \times 1.\]

Every element $b_i \in F_d^{ab} \times 1$ is a product of at most $d$ palindromes and every $a_i$ can be written as a product of
at most
$\max\left\{l_X(h) \h | \h h \in H\right\}$ palindromes. We do this for every element of $H$, hence we need at most
$\max\left\{l_X(h) \h | \h h \in H\right\} \cdot |H| \cdot d$ palindromes for $g \in F_d' \wr H$ and at most $\max\left\{l_X(h)
\h | \h h \in H\right\} + 1$ palindromes
for $t \in F_d^{ab} \wr H$, which gives at most \[\max\left\{l_X(h) \h | \h
h \in H\right\} \cdot (d \cdot |H| + 1)+1\] palindromes for $g \cdot t$. The statement follows since every element of $F_d \wr H$
can be expressed in that way.
\end{proof}

By Lemma \ref{lemma_pw_in_quotients}, this also carries over to the wreath product of any finitely generated group $G$ with $H$ as
above.

\begin{cor}
Assume that $H$ is a non-abelian finite group with generating set $X$ and $G$ a group with $d$ generators. Then
\[pw\left(G \wr H, X \cup \{c\} \cup Y\right) \leq \max\left\{l_X(h) \h | \h h \in H\right\} \cdot (d \cdot |H| + 1) +
1,\] with $c \in H$, possibly $c=1$, for any generating set $Y$ of $G$.
\end{cor}

\renewcommand{\bibname}{References}
\bibliography{bibliography}        %use a bibtex bibliography file refs.bib
\bibliographystyle{plain}  %use the plain bibliography style

\end{document}